\documentclass[12pt]{amsart}

\usepackage{latexsym}
\usepackage{enumerate}
\usepackage{amsmath, amssymb, amsfonts, amsxtra, amsthm}

\usepackage[normalem]{ulem}

\usepackage{amssymb}
\allowdisplaybreaks
\numberwithin{equation}{section}

\newcommand {\rea}{\mathbb{R}}

\def\P{{\bf P}}
\def\Q{{\bf Q}}
\def\T{{\bf T}}
\def\<{\left<}
\def\>{\right>}
\newcommand {\size}{\mathrm{size}}
\newcommand {\sgn}{\mathop{\mathrm{sgn}}}
\newcommand {\ov}{\overline}
\newcommand {\calI}{\mathcal{I}}
\newcommand {\calJ}{\mathcal{J}}

\newcommand {\tphi}{\tilde{\phi}}

\newtheorem{theorem}{Theorem}[section]

\newtheorem{proposition}[theorem]{Proposition}
\newtheorem{claim}[theorem]{Claim}
\newtheorem{lemma}[theorem]{Lemma}

\title[Variational bounds for the quartile operator]{Variational bounds for a dyadic model of the bilinear Hilbert transform}

\author[Y. Do, \ \ R. Oberlin, \ \  E. A. Palsson]
{Yen Do \ \ \ Richard Oberlin \ \ \ Eyvindur Ari Palsson}
\address{Department of Mathematics,
Yale University, New Haven, CT 06511, USA}
\email{yen.do@yale.edu}

\thanks{The second author is supported in part by NSF Grant DMS-1068523.}
\address{Department of Mathematics,
Louisiana State University,
Baton Rouge, LA 70803-4918, USA}
\email{oberlin@math.lsu.edu}

\address{Department of Mathematics, University of Rochester, Rochester, NY 14627-0251, USA}
\email{palsson@math.rochester.edu}

\begin{document}
\begin{abstract}{We prove variation-norm estimates for the Walsh model of the truncated bilinear Hilbert transform, extending related results of Lacey, Thiele, and Demeter. The proof uses analysis on the Walsh phase plane and two new ingredients: (i) a variational extension of a lemma of Bourgain by Nazarov--Oberlin--Thiele, and (ii) a variation-norm Rademacher--Menshov theorem of Lewko--Lewko.}
\end{abstract}

\maketitle

\section{Introduction}
In this paper we consider a variation-norm analog of the following maximal operator
\[
H^*[f_1,f_2](x) = \sup_k |\sum_{|I_P| \geq 2^k} |I_P|^{-1/2}\<f_1,\phi_{P_1}\>\<f_2,\phi_{P_2}\>\phi_{P_3}(x)|
\]
where we sum over $P$ in a collection $\ov{\P}$ of dyadic rectangles in $\mathbb R^+ \times \mathbb R^+$ of area four (also known as quartiles) and $\phi_{P_1},\phi_{P_2},\phi_{P_3}$ denote dyadic wave packets adapted to appropriate subsets of $P$, see Section \ref{termsection} for details. Note that we have suppressed the notational dependency on $\ov{\P}$ for simplicity (and all implicit constants in this paper shall be independent of the underlying collection of quartiles). The non-maximal variant of $H^*$ is known as the quartile operator and was introduced in \cite{thiele95tfa} as a discrete model of the bilinear Hilbert transform. The operator $H^*$ serves as a dyadic model for both the maximal bilinear Hilbert transform and the bilinear maximal function \cite{lacey00bmf} (cf. \cite{thiele01mqo, demeter07pce}). See also the discussion after \eqref{e.variationBHT}.

Our aim here is to bound the operator formed by replacing the $\ell^{\infty}$ norm in the definition of $H^*$ by a stronger variation-semi-norm. Given an exponent $r \geq 1$ write
\[
\|g\|_{V^r_k} = \sup_{N, k_0 < \cdots < k_N} (\sum_{j = 1}^N|g(k_j) - g(k_{j-1})|^r)^{1/r}
\]
where the supremum is over all strictly increasing finite-length sequences of integers. Setting
\[
H^r[f_1,f_2](x) = \|\sum_{|I_P| \geq 2^k} |I_P|^{-1/2}\<f_1,\phi_{P_1}\>\<f_2,\phi_{P_2}\>\phi_{P_3}(x)\|_{V^r_k},
\]
we will prove
\begin{theorem} \label{maintheorem} Suppose $r > 2$, and $p_1, p_2, q$ satisfy
\[
\frac{1}{q} = \frac{1}{p_1} + \frac{1}{p_2}, \ \ \ \ \frac{2}{3} < q < \infty, \ \ \ \ 1 < p_1, p_2 \le \infty,
\] 
then for some constant $C=C(p_1,p_2,r)<\infty$ we have
\begin{equation} \label{maintheoremeq}
\|H^r[f_1,f_2]\|_{L^q} \leq C \|f_1\|_{p_1} \|f_2\|_{p_2}.
\end{equation}
\end{theorem}

We became interested in bounds for $H^r$ while studying the following bilinear operator
\begin{equation}\label{e.variationBHT}
B^r[f_1,f_2](x) = \|\frac{1}{2t}\int_{-t}^tf_1(x + y)f_2(x-y)\ dy \|_{V^r_t}.
\end{equation}
The simpler maximal variant of $B^r$ is the bilinear maximal function studied in \cite{lacey00bmf}. An oscillation-norm variant\footnote{More precisely, \cite{demeter07pce} proves  estimates on a finite sum of oscillations, however the implicit constant depends on the number of oscillations being measured.} of $B^r$ was also considered in \cite{demeter07pce}. Bounds for the simpler linear version of $B^r$, i.e. the variation-norm analog of the centered Hardy-Littlewood maximal function, proved in \cite{bourgain89pes}, can be used to strengthen the Birkhoff ergodic theorem on the pointwise convergence of linear ergodic averages. Similarly, bounds on the more delicate $B^r$ and its oscillation-norm variants are useful for studies of pointwise convergence of bilinear ergodic averages, see e.g. \cite{demeter07pce}. While the oscillation-norm estimates in \cite{demeter07pce} are just enough for this purpose, bounds on $B^r$ give more quantitative information about the related rate of convergence.  In this paper, only the dyadic variant $H^r$ of $B^r$ will be considered, which is technically simpler than the continuous setting and therefore allows for a relatively clear and accessible illustration of the main ideas, which we expect to be useful in forthcoming study of $B^r$.   

\subsection{Structure of the paper} We essentially follow the framework of \cite{lacey00bmf}, \cite{thiele01mqo}, although the argument is slightly reorganized and simplified to reflect the modern language of time-frequency analysis. The main new ingredients in the proof are variational extensions of several maximal theorems, including a variation-norm extension of the Rademacher--Menshov theorem obtained in \cite{lewko12esv} and an extension of a lemma of Bourgain \cite{bourgain89pes} to the variation-norm setting obtained in \cite{nazarov10czd} (cf. \cite{oberlin11awm}). These auxiliary results and other background materials are summarized in Sections \ref{termsection} and \ref{auxsection}. Several technical lemmas are proven in Sections \ref{sbsection} and \ref{sisection}, and we show how they imply Theorem~\ref{maintheorem} in Section \ref{poftsection}. 

%The parts of the argument which are sensitive to the distinction between $H^*$ and $H^r$ are Propositions \ref{propositionvnsb} and \ref{mainpropositionweak}. 

\subsection{Notational conventions}
We use $|\cdot|$ to denote Lebesgue measure, cardinality, or an understood norm depending on context. The indicator function of a set $E$ is written $1_E.$ Dyadic intervals are half-open on the right, i.e. of the form $[n2^k, (n+1)2^k)$ for integers $n,k.$

\subsection*{Acknowledgement} This work was initiated while the authors were visiting the University of California, Los Angeles in Winter 2012, and the visit was supported in part by the AMS Math Research Communities program. The authors would like to thank the MRC and Christoph Thiele for their generous support, hospitality, and useful conversations.

\section{Terminology} \label{termsection}

In this paper, a \emph{tile} is a dyadic rectangle in $\rea^+ \times \rea^+$ of area one. A \emph{quartile} is defined analogously, except with area four instead of one. Each quartile $P=I_P \times \omega_P$ can be written as the disjoint union of four tiles $P_1, P_2, P_3, P_4$ where $P_i = I_{P_i} \times \omega_{P_i},$ $I_{P_i} = I_P$, and $\omega_{P_i}$ is the i'th dyadic grandchild of $\omega_P$, in increasing order from left to right.

The Walsh wave-packet $\phi_p$ associated to a tile $p$ can be defined as follows. First, if $p = I \times [0,2^{n})$ then $\phi_{p}(x) = 2^{-n/2}1_{I}(x).$ To extend the definition to all tiles, we use the following recursive formulas where the subscripts $l$ and $r$ denote the left and right halves of a dyadic interval:
\begin{equation} \label{recursive}
\phi_{I \times \omega_r} = \frac{1}{\sqrt{2}}(\phi_{I_l \times \omega} - \phi_{I_r \times \omega}), \ \ \phi_{I \times \omega_l} = \frac{1}{\sqrt{2}}(\phi_{I_l \times \omega} + \phi_{I_r \times \omega}).
\end{equation}
It is not hard to see that $\phi_p$ and $\phi_{p'}$ are orthogonal if $p \cap p' = \emptyset$.

Given a collection of quartiles $T$, a ``top frequency'' $\xi_T \in \rea^+$, and a dyadic ``top interval'' $I_T \subset \rea^+$, we say that $(T,\xi_T, I_T)$ form a \emph{tree} if for every $P \in T$, $I_P \subset I_T$ and $\xi_T \in \omega_P$. Letting $\omega_T$ be the dyadic interval of length $|I_T|^{-1}$ containing $\xi_T$, we write $p_T$ for the tile $I_T \times \omega_T.$

For $i=1, \ldots, 4$ a tree is said to be \emph{$i$-overlapping} if for every $P \in T$, $\xi_T \in \omega_{P_i}.$ We will say that a tree is \emph{$i$-lacunary} if it is $j$-overlapping for some $j \neq i.$ One can check that if $T$ is $i$-lacunary then the tiles $\{P_i\}_{P \in T}$ are pairwise disjoint.

To define a notion of size that is compatible with $H^r$, we will need to linearize and dualize the variation-norm. For each $x$ consider an increasing integer-valued sequence $\{k_j(x)\}_{j = -\infty}^{\infty}$, and a sequence $\{a_j(x)\}_{j = -\infty}^{\infty}$ such that $\sum_{j = -\infty}^{\infty}|a_j(x)|^{r'} \leq 1$. Then an appropriate choice of such sequences guarantees that, for every $x$,
\begin{equation}\label{modelop}
\sum_{P \in \ov{\P}} |I_P|^{-1/2}\<f_1,\phi_{P_1}\>\<f_2,\phi_{P_2}\>\phi_{P_3}(x)a_P(x) \geq \frac{1}{2} H^r[f_1,f_2](x) 
\end{equation}
where $a_P(x) := a_m(x)$ if $m=m(P,x)$ is the (clearly unique) integer satisfying $2^{k_{m-1}(x)} \leq |I_P| < 2^{k_{m}(x)}$, and $0$ if such $m$ does not exist. 

Thus, to prove \eqref{maintheoremeq}, it suffices to give a corresponding bound for the left side above which is independent of the choice of sequences. Fixing these sequences once and for all, we write
\[
\tphi_{P_3}(x) = a_P(x) \phi_{P_3}(x), 
\]
and when $i\neq3$ write $\tphi_{P_i}(x) = \phi_{P_i}(x).$\footnote{At first glance this notation might seem slightly abusive, but since each tile is contained in a unique quartile, $\tilde{\phi}_p$ is defined implicitly for any tile $p$.}
For collections of quartiles $\P$ and functions $f$ on $\rea^+$ we define
\begin{equation}\label{sizedef}
\size_i(\P,f) = \sup_{T \subset \P} (\frac{1}{|I_T|}\sum_{P \in T}|\<f,\tphi_{P_i}\>|^2)^{1/2}
\end{equation}
where the supremum is over all $i$-lacunary trees contained in $\P.$

Let $A_k$ denote the dyadic averaging operator
\[
A_k[f](x) = \frac{1}{|I|} \int_{I} f(y)\ dy
\]
where $I$ is the unique dyadic interval of length $2^k$ containing $x$. Sums of wave packets in lacunary trees can be truncated using $A_k$ as follows:
\begin{claim} \label{averageclaim}
Suppose that $T$ is a $i'$-overlapping tree. Then for $i \neq i'$ there is $\nu\in \{0,1\}$ such that for any coefficients $\{c_P\}_{P \in T}$
\begin{equation} \label{trunctree}
\sum_{P \in T: \, \, |I_P| > 2^{k + \nu} } c_P \phi_{P_i} = \sgn(\phi_{p_T}) A_{k}[\sgn(\phi_{p_T} ) \sum_{P \in T} c_P \phi_{P_i}].
\end{equation}
Moreover, $\nu = 0$ when $\{i,i'\}=\{1,2\}$ or $\{3,4\}$, and $\nu = 1$ otherwise.
\end{claim}

\begin{proof}
If $\nu = 0$ one can check using \eqref{recursive} that 
\[
\sgn(\phi_{p_T})\phi_{P_i} = \pm|I_P|^{-1/2}(1_{(I_{P})_l} - 1_{(I_{P})_r}) 
\]
where choice of sign $\pm$ is uniform over $x$. If $\nu = 1$ then similarly
\[
\sgn(\phi_{p_T}) \phi_{P_i} = \pm|I_P|^{-1/2}\left((1_{((I_{P})_l)_l} - 1_{((I_{P})_l)_r}) \pm(1_{((I_{P})_r)_l} - 1_{((I_{P})_r)_r})\right).
\] 
The claim then follows by inspection of averages.
\end{proof}

A consequence of \eqref{trunctree} is that for each $x$
\begin{equation} \label{vartrunc}
\|\sum_{P \in T: \, \, |I_P| \geq 2^{k} } c_P \phi_{P_i}(x)\|_{V^r_k} = \|A_k[\sgn(\phi_{p_T} ) \sum_{P \in T} c_P \phi_{P_i}](x)\|_{V^r_k}.
\end{equation}

\section{Auxiliary estimates} \label{auxsection}

We will start by recalling three variation-norm bounds which will be of later use. The first is a special case of a theorem of L\'{e}pingle \cite{lepingle76lvd}. 
\begin{lemma} \label{leplemma}
Suppose $r > 2$ and $1 < t < \infty$. Then
\[
\|A_k[f](x)\|_{L^t_x(V^r_k)} \leq C_{r,t} \|f\|_{L^t}.
\]
\end{lemma}

Now, let $\Xi$ be any finite subset of $\rea^+$, and for each integer $k$ let $\Omega_k$ be the set of dyadic intervals of length $2^{-k}$ which intersect $\Xi.$ Let
\begin{equation}\label{mfprojection}
\Delta_k[f] = \sum_{\omega \in \Omega_k} \sum_{I : \,\, |I| = 2^k} \<f,\phi_{I \times \omega}\>\phi_{I \times \omega}.
\end{equation}
Note that while the definition of $\Delta_k[f]$ involves an infinite sum, for each $x$ only finitely many terms are nonzero. Equivalently,
\[
\Delta_k[f] = \sum_{\omega \in \Omega_k} (1_{\omega} \hat{f})\check{\ }.
\]
The following Lemma follows from \cite[Lemma 9.2]{oberlin11awm}, see also \cite{nazarov10czd}, which is a variation-norm extension of a lemma of Bourgain \cite{bourgain89pes}.
\begin{lemma} \label{vblemma}
Suppose $r > 2$ and $\epsilon > 0.$ Then
\[
\|\Delta_k[f](x)\|_{L^2_x(V^r_k)} \leq C_{r,\epsilon} |\Xi|^{\epsilon} \|f\|_{L^2}.
\]
\end{lemma}

Below, we have a variation-norm Rademacher-Menshov theorem which was proven in \cite{lewko12esv}, see also the proof of Theorem 4.3 in \cite{nazarov10czd}.
\begin{lemma} \label{vnrmlemma}
Let $X$ be a measure space and $f_1, \ldots, f_N$ be orthogonal functions on $X$. Then
\[
\|\sum_{j = 1}^n f_j(x)\|_{L^2_x(V^2_n)} \leq C (1 + \log(N)) (\sum_{j = 1}^{N}\|f_j\|^2_{L^2})^{1/2}.
\]
\end{lemma}

Finally we will need the following John-Nirenberg type lemma. See, for example, the proof of Lemma 4.2 in \cite{muscalu04leb}.

\begin{lemma} \label{jnlemma}
Let $\{c(P)\}_{P \in \P}$ be a collection of coefficients. Denote
\[
A_{1,\infty} = \sup_{T \subset \P} \frac{1}{|I_T|}\|(\sum_{P \in T}|c(P)|^2\frac{1_{I_P}}{|I_P|})^{1/2}\|_{L^{1,\infty}},
\]
\[
A_2 = \sup_{T \subset \P} \frac{1}{|I_T|^{1/2}}\|(\sum_{P \in T}|c(P)|^2\frac{1_{I_P}}{|I_P|})^{1/2}\|_{L^{2}}
\]
where both supremums are over (say) all $i$-lacunary trees. Then
\begin{equation} \label{jnconc}
A_2 \leq C A_{1,\infty}.
\end{equation}
\end{lemma}

\section{A variation-norm size bound} \label{sbsection}

Let $M^t$ denote the dyadic $L^t$-Hardy-Littlewood maximal operator
\begin{equation}\label{e.dyadicHL}
M^t[f](x) = \sup_k (A_k[|f|^t](x))^{1/t}.
\end{equation}

\begin{proposition} \label{propositionvnsb}
Let $\lambda > 0, r > 2,$ and $1 < t < \infty$. Suppose that $\P$ is a collection of quartiles such that for each $P \in \P$
\begin{equation} \label{notbadtileeq}
I_P \not\subset \{M^t[f] > \lambda\}.
\end{equation}
Then for each $i$
\[
\size_i(\P,f) \leq C_{r,t} \lambda.
\]
\end{proposition}

\begin{proof}
Using Lemma \ref{jnlemma}, it follows from \eqref{sizedef} that
\[
\size_i(\P,f) \leq C \sup_{T \subset \P} |I_T|^{-1/t} \|(\sum_{P \in T}|\<f,\tphi_{P_i}\>|^2\frac{1_{I_P}}{|I_P|})^{1/2}\|_{L^t}
\]
which, by the usual Rademacher function argument, is
\[
\leq C_t \sup_{T \subset \P} \sup_{\{b_P\}_{P \in \P}}|I_T|^{-1/t} \|\sum_{P \in T} b_P \<f,\tphi_{P_i}\> \phi_{P_i} \|_{L^t}
\]
where the right supremum is over all sequences $\{b_P\}_{P \in \P}$ of $\pm1$'s.

Let $t' = t/(t - 1)$. For any nonempty $i$-lacunary tree $T\subset \P$ and any binary sequence $\{b_P\}$, by duality we have
\begin{align*}
&|I_T|^{-1/t} \|\sum_{P \in T} b_P \<f,\tphi_{P_i}\> \phi_{P_i} \|_{L^t} \\
&\leq |I_T|^{-1/t} \|1_{I_T}f\|_{L^t} \sup_{g : \|g\|_{L^{t'} = 1}} \|\sum_{P \in T} b_P \<g,\phi_{P_i}\> \tphi_{P_i} \|_{L^{t'}}  \\
&\leq \lambda \sup_{\|g\|_{L^{t'} = 1}} \|\sum_{P \in T} b_P \<g,\phi_{P_i}\> \tphi_{P_i} \|_{L^{t'}}   \qquad \qquad \text{(using \eqref{notbadtileeq})} \\
&\leq \lambda \sup_{\|g\|_{L^{t'} = 1}} \|\sum_{P \in T: \,\, |I_P| \geq 2^k} b_P \<g,\phi_{P_i}\> \phi_{P_i}(x) \|_{L^{t'}_x(V^r_k)} \qquad \text{(by def. of $\tphi_{P_i}$)}\\
&= \lambda \sup_{\|g\|_{L^{t'} = 1}} \|A_k[\sgn(\phi_{p_T})\sum_{P \in T} b_P \<g,\phi_{P_i}\> \phi_{P_i}](x) \|_{L^{t'}_x(V^r_k)}  \qquad \text{(by \eqref{vartrunc})} \\
&\leq C_{r,t'}\lambda \sup_{\|g\|_{L^{t'} = 1}} \|\sum_{P \in T} b_P \<g,\phi_{P_i}\> \phi_{P_i}\|_{L^{t'}}  \qquad \text{(by Lemma \ref{leplemma})} \\
&\leq C_{r,t'} \lambda \qquad \text{(by standard dyadic Calder\'{o}n-Zygmund theory).} 
\end{align*}
Note that much of the argument above is superfluous unless $i=3$.
\end{proof}

\section{A variation-norm size lemma} \label{sisection}
The main result in this section is Proposition~\ref{vnsprop}, and its proof requires Propositions~\ref{mainpropositionstrong}  and \ref{mainpropositionweak}. We assume throughout this section that $\epsilon>0$ and $r>2$, and all implicit constants are allowed to depend on $\epsilon$ and $r$.
\begin{proposition} \label{vnsprop}
Let $\P$ be a finite collection of quartiles. Suppose $|f| \leq 1_E.$ Then for each $\alpha$ satisfying
\[
\size_i(\P, f)^2 \leq \alpha
\]
we can find a collection of trees $\T$, each contained in $\P$, satisfying 
\[
\size_i(\P \setminus \bigcup_{T \in \T}T, f)^2 \leq \frac{1}{4} \alpha,
\]
\begin{equation} \label{vnslconcNeq}
\sum_{T \in \T} |I_T| \leq C  \alpha^{-(1+ \epsilon)} |E|.
\end{equation}
\end{proposition}

\begin{proof}
Let $j_1, j_2, j_3$ be an enumeration of $\{1,2,3,4\} \setminus \{i\}$
and $\P_0^1 = \P$. If there is a $j_1$-overlapping tree $S \subset \P^1_0$ satisfying 
\begin{equation} \label{biglac}
\frac{1}{|I_{S}|}\sum_{P \in S}|\<f,\tphi_{P_i}\>|^2 \geq \frac{1}{4} \alpha 
\end{equation}
then let $S_1^1$ be such a tree, chosen in the following manner:\\
\indent (i) If $j_1 < i$ then we pick such $S_1^1$ with $\inf \omega_{S_1^1}$ maximal. \\
\indent (ii) If $j_1 > i$ then we pick such $S_1^1$ with $\inf \omega_{S_1^1}$ minimal.\\   
We then let $T_1^1$ be the maximal (with respect to inclusion) tree contained in $\P^1_0$ with top data $(I_{S_1^1},\xi_{S_1^1})$. 

Now, consider $\P_1^1 = \P_0^1 - T_1^1$ and iterate the above selection process until no trees satisfying \eqref{biglac} can be found (the process must stop in finite time due to the assumption that $\P$ is finite), we obtain trees $T_1^1, \ldots, T_{n_1}^1$ and $S_1^1, \ldots, S_{n_1}^1$ where $S_1^1, \ldots, S_{n_1}^1$ are $j_1$-overlapping trees and satisfy $\eqref{biglac}$, and with $S_{i}^1 \subset T_{i}^1$ for each $i$ .  

We then consider the remaining tile collection $\P_0^2 = \P_0^1 - T_1^1 - \dots - T_{n_1}^1$ and repeat the same process as above, but choosing $j_2$-overlapping trees instead of $j_1$-overlapping trees. This gives $S_1^2, \ldots, S_{n_2}^2$ which are inside more general trees $T_1^2, \ldots, T_{n_2}^2$. Finally we select $j_3$-overlapping trees from the remaining tile collection and  obtain $j_3$-overlapping trees $S_1^3, \ldots, S_{n_3}^3$ which are inside trees $T_1^3, \ldots, T_{n_3}^3$.

By construction we have
\[
\size_i(\Q, f)^2 \leq \frac{1}{4} \alpha 
\qquad \text{where} \qquad \Q = \P \setminus \bigcup_{k = 1,2,3} \bigcup_{l = 1}^{n_k}T_{l}^k.
\]
Thus, it remains to show that for each $k$
\[
\sum_{l=1}^{n_k} |I_{S_l^k}| \leq C \alpha^{-(1 + \epsilon)} |E|.
\]
To verify this estimate, first note that the sets $\bigcup_{P \in S_l^k} P_i$ indexed by $l$ are pairwise disjoint. Indeed, suppose $P_i \cap P'_i \neq \emptyset$, $P \in S_l^k,$ $P' \in S_{l'}^k$, and $l < l'$. By geometry, the maximality/minimality of $\inf \omega_{S_l^k}$ guarantees that $P' \in T_l^k$, contradicting the fact that $P' \in S_{l'}^k.$ 

Now, if $i\ne 3$ then by orthogonality of the $\tphi_{P_i}$ and \eqref{biglac} we have
\[
\sum_{l=1}^{n_k} |I_{S_l^k}| \leq C \alpha^{-1} \|f\|_{2}^{2}
\]  
which implies \eqref{vnslconcNeq} (the assumption that $f \leq 1_E$ guarantees that any $\size_i(\P,f) \leq 1$ and so we may assume that $\alpha \leq 4$).

If $i=3$, orthogonality between $\tphi_{P_3}$ is not available, and we apply Proposition \ref{mainpropositionstrong} below. 
\end{proof}

\begin{proposition} \label{mainpropositionstrong} 
Suppose that $\T$ is a finite collection of 3-lacunary trees such that
the elements of $\{P_3:\, P\in \bigcup_{T \in \T} T\}$ are pairwise disjoint
and furthermore for each $T \in \T$
\begin{equation} \label{largesizeineq}
\frac{1}{|I_T|} \sum_{P \in T } |\<f,\tphi_{P_3}\>|^2 \geq \alpha.
\end{equation}
Then for $N := \sum_{T \in \T}1_{I_T}$ we have
\begin{equation} \label{mainpropstrongconc}
\|N\|_{L^1} \leq C  \alpha^{-(1+\epsilon)} \|f\|^{2 + 2\epsilon}_{L^{2 + 2\epsilon}}.
\end{equation}
\end{proposition}
\begin{proof}
We'll show that if $c>0$ is sufficiently small then for $\lambda \geq 1$ 
\begin{equation} \label{goodlambda}
|\{N > \lambda\}| \leq |E_\lambda| + \frac{1}{100} |\{N > \lambda/4\}|\ , \qquad \text{where}
\end{equation}
$$E_\lambda:=\{M^2 f > c \alpha^{1/2} \lambda^{1/(2 + 2\epsilon)}\},$$
and $M^2f$ is the $L^2$ dyadic Hardy-Littlewood maximal function (see \eqref{e.dyadicHL}). Once this is done, we can integrate both sides of \eqref{goodlambda},
\begin{align*}
\|N\|_{L^1} &\leq \int |E_\lambda| \ d\lambda + \int \frac{1}{100} |\{N > \lambda/4\}|  d\lambda   \\
&= C\, \alpha^{-(1 + \epsilon)}\|M^2 f \|_{2 + 2\epsilon}^{2 + 2\epsilon} + \frac{1}{25} \|N\|_{L^1} \\
&\leq  C\, \alpha^{-(1 + \epsilon)}\|f\|_{2 + 2\epsilon}^{2 + 2\epsilon} + \frac{1}{25} \|N\|_{L^1},
\end{align*}
and obtain the desired claim \eqref{mainpropstrongconc}.

Let $\calI$ be the collection of maximal dyadic intervals contained in $\{N > \lambda/4\}$. This collection clearly covers $\{N > \lambda\}$. Thus, \eqref{goodlambda} will follow if for any $I \in \calI$ that intersects the set $E_\lambda$ it holds that 
\[
|\{N > \lambda \} \cap I| \leq \frac{1}{100}|I|.
\]
To see this, take $I$ be such an interval. Then
\[
\|1_I f\|_{L^2} \le |I|^{1/2} \inf_{x\in I} M^2[f](x)
\]
\begin{equation}\label{e.localL2} 
\leq |I|^{1/2} c \alpha^{1/2} \lambda^{1/(2 + 2\epsilon)}.
\end{equation}
It follows from the maximality of $I$ that
\[
\{N > \lambda\} \cap I \subset \{N_I > \lambda/4\} \qquad \text{where} \qquad N_I := \sum_{T \in \T: \, \, I_T \subset I} 1_{I_T} .
\]
Finally, applying Proposition \ref{mainpropositionweak} with $1_If$ in place of $f$ and $\{T \in \T : I_T \subset I\}$ in place of $\T$, we obtain for some $C'$ depends on $r,\epsilon'$:
\begin{align*}
|\{N_I \geq \lambda/4\}| & \leq C' \, \alpha^{-1} \lambda^{-(1-\epsilon')} \|1_If\|_{L^2}^2 \\
&\leq C' \, \alpha^{-1} \lambda^{-(1-\epsilon')} \,\, c^2 \alpha \lambda^{\frac{1}{1 + \epsilon}}  \,\, |I| \qquad \text{(by \eqref{e.localL2})} \\
&\leq |I|/100
\end{align*}
where the last inequality follows by choosing $\epsilon' = \frac{\epsilon}{1 + \epsilon}$, and a sufficiently small choice of $c$ depending on $C'$.
\end{proof}

\begin{proposition} \label{mainpropositionweak} 
Suppose that $\T$ and $N$ are as in the hypotheses of Proposition \ref{mainpropositionstrong}. 
Then
\begin{equation} \label{mainpropconc}
|\{x : N(x) > \lambda\}| \leq C \alpha^{-1} \frac{\|f\|^2_{L^2}}{\lambda^{1 - \epsilon}}.
\end{equation}
\end{proposition}

\begin{proof}
We may assume $\lambda \geq 1$ as $N$ is integer valued. We first estimate
\begin{equation}\label{suminl}
|\{x : N(x) > \lambda\}|  \leq \sum_{l \geq 0} |\{2^l \lambda < N \leq 2^{l+1} \lambda\}|.
\end{equation}
It is clear that trees $T$ with $I_T \subset \{N > 2^{l+1} \lambda\}$ make no contribution to the $l$'th set in the display above. Thus, letting
\[
\T_l := \{T \in \T : I_T \not\subset \{N > 2^{l+1} \lambda\}\}
\]
and $N_l = \sum_{T \in \T_l}1_{I_T}$, we have
\[
\{2^l \lambda < N \leq 2^{l+1} \lambda\} \subset \{N_l > 2^l \lambda\}.
\]
Note that in the evaluation of $N_l$ at each point, only a nested sequence of top intervals are involved, and the smallest of them intersects $\{N\le 2^{l+1}\lambda\}$. It follows that
\begin{equation} \label{linftyNl}
\|N_{l}\|_{L^{\infty}} \leq 2^{l+1} \lambda.
\end{equation}

By Chebyshev and \eqref{largesizeineq} , with $\P_l = \bigcup_{T \in \T_l} T$ we have
\[
|\{N_l > 2^l \lambda\}| \leq (\alpha 2^l \lambda )^{-1} \sum_{P \in \P_l}|\<f,\tphi_{P_3}\>|^2.
\]
Consequently, together with \eqref{suminl}, it will suffice for \eqref{mainpropconc} to show
\[
\|\sum_{P \in \P_l}\<f,\tphi_{P_3}\>\phi_{P_3}\|_{L^2} \leq C (2^l\lambda)^{\epsilon} \|f\|_{L^2}.
\]
Invoking duality and unravel the definition of $\tphi_{P_3}$, this follows from
\begin{equation} \label{unlinearized}
\|\sum_{P \in \P_l: \,\, |I_P| \geq 2^k} \<f,\phi_{P_3}\>\phi_{P_3}(x)\|_{L^2_x(V^r_k)} \leq C (2^l\lambda)^{\epsilon} \|f\|_{L^2}.
\end{equation}
In the rest of the proof, we show \eqref{unlinearized}. 

We will repeatedly use a ``long-jump/short-jump'' decomposition to estimate the variation-norm. Namely, if $\{k_j\}$ is any strictly increasing sequence of integers, then for $r \geq 2$ we have
\begin{equation} \label{longshortbound}
\|g(k)\|_{V^r_k} \leq C(\|g(k_j)\|_{V^r_j} + \|g(k)\|_{\ell^2_j(V^r_{k_j \leq k < k_{j+1}})})   
\end{equation}
where the notation $\|\cdot\|_{V^r_{a \leq k < b}}$ indicates that, in the variation-norm, we only consider sequences lying between $a$ and $b$ (in the case that $b=a+1$, we set the value to $0$).

The first step in proving \eqref{unlinearized} is the following decomposition of $\P_l$.\\
(i) Let $\calJ_1$ be the collection of top intervals of elements of $\T_l$. \\
(ii) For $m \geq 1$ let $\calI_m$ be the set of maximal intervals in $\calJ_m$.\\
(iii) Let $\calJ_{m+1} = \calJ_m \setminus \calI_m.$ \\
(iv) Let $\P_{l,m} \subset \P_l$ contains those quartiles $P$ such that there exists an element of $\calI_m$ that contains $I_P$, but no such exists in $\calI_{m+1}$.

It then follows from \eqref{linftyNl} that $\calJ_m = \calI_m = \emptyset$ for $m > 2^{l+1} \lambda$. Furthermore, every $x$ that contributes to the left hand side of \eqref{unlinearized} is contained in a chain $I_1 \supset \cdots \supset I_{M(x)}$ of intervals, where $I_i \in \calI_i$. Forming a sequence $\{k_j(x)\}$ based on the lengths of these intervals, and applying \eqref{longshortbound} pointwise, it follows that the left side of \eqref{unlinearized} is 
\begin{equation} \label{unlinearizedsplit}
\leq C \|\sum_{m = 1}^n \sum_{P \in \P_{l,m}} \<f,\phi_{P_3}\>\phi_{P_3}(x)\|_{L^2_x(V^r_{n \leq 2^{l+1} \lambda})}
\end{equation}
\begin{equation} \label{unlinearizedsplit2}
+ C \|\sum_{P \in \P_{l,m}: \,\, |I_P| \geq 2^k} \<f,\phi_{P_3}\>\phi_{P_3}(x)\|_{L^2_x(\ell^2_m(V^r_k))}.
\end{equation}
By Lemma \ref{vnrmlemma}, we can bound \eqref{unlinearizedsplit} by
\[
\leq C (1 + \log(2^{l}\lambda)) (\sum_{m=1}^{[2^{l+1}\lambda]} \| \sum_{P \in \P_{l,m}} \<f,\phi_{P_3}\>\phi_{P_3}\|_{L^2}^2)^{1/2}.
\]
It then follows from disjointness of $P_3$'s that the above display is controlled by the right hand side of \eqref{unlinearized}.

It remains to bound \eqref{unlinearizedsplit2}. For each  $I\in \calI_m$ let $\P(I)$ be the set of elements of $\P_{l,m}$ whose time interval is inside $I$. Note that these collections form a partition of $\P_{l,m}$. Using Fubini and spatial orthogonality, we can rewrite \eqref{unlinearizedsplit2} as
\begin{equation}\label{unlinearizedsplit3}
C (\sum_{m=1}^{2^{l+1} \lambda} \sum_{I \in \calI_m}\|\sum_{P \in \P(I): \,\, |I_P| \geq 2^k} \<f,\phi_{P_3}\>\phi_{P_3}(x)\|_{L^2_x(V^r_k)}^2)^{1/2}.
\end{equation}
It suffices to show that, for each $m$ and each $I\in \calI_m$, 
$$\|\sum_{P \in \P(I): \,\, |I_P| \geq 2^k} \<f,\phi_{P_3}\>\phi_{P_3}(x)\|_{L^2_x(V^r_k)}$$
\begin{equation}\label{localunlinearized}
\le C(2^l\lambda)^\epsilon \|\sum_{P \in \P(I)} \<f,\phi_{P_3}\>\phi_{P_3}(x)\|_{L^2_x}
\end{equation}
(the desired bound for \eqref{unlinearizedsplit3} then follows from disjointness of $P_3$'s). 

Let $\T(I)$ be the set of trees in $\T_l$ that intersect $\P(I)$, then 
\begin{equation}\label{localtreecount}
|\T(I)| \leq 2^{l+1}\lambda.
\end{equation}
Indeed, the top interval of any $T \in \T_I$ must contain $I$, since otherwise it would be contained in some element of $\calI_{m+1}$ and so $T\cap\P_{l,m}=\emptyset$, contradiction. Consequently, $|\T(I)|$ is equal to the value of $\sum_{T \in \T(I)}1_T$ on $I$. The above estimate then follows from \eqref{linftyNl}.

Now, by splitting $\T(I)$ and absorbing a factor, we may assume that there is an $i' \ne 3$ such that every element of $\T(I)$ is $i'$-overlapping. Let $\nu$ be as in Claim \ref{averageclaim} with $i=3$. 

Form $\Omega_j$ and $\Delta_k$ as in \eqref{mfprojection} using the collection $\Xi$ of top frequencies of elements of $\T(I)$. By \eqref{localtreecount} we have
$|\Xi| \leq 2^{l+1}\lambda$.

Let $k_1, \ldots, k_N$ be the increasing enumeration of those $k$'s such that 
$$|\Omega_{k+4}| > |\Omega_{k-4}|.$$
Since $|\Omega_k|\le |\Xi|$, it follows that $N \leq 8|\Xi|\le 2^{l+4}\lambda$. 

\emph{Below and for the rest of the current proof, all quartiles are in $\P(I)$, in particular in the summations.} Applying \eqref{longshortbound} we have
$$\|\sum_{|I_P| \geq 2^k} \<f,\phi_{P_3}\>\phi_{P_3}(x)\|_{L^2_x(V^r_k)} \leq C \|\sum_{|I_P| \geq 2^{k_j}} \<f,\phi_{P_3}\>\phi_{P_3}(x)\|_{L^2_x(V^r_j)}$$
\begin{equation} \label{longshortagain}
 \qquad + \, C\, \|\sum_{2^k  \leq |I_P| < 2^{k_{j+1}}} \<f,\phi_{P_3}\>\phi_{P_3}(x)\|_{L^2_x(\ell^2_j(V^r_{k_{j} \leq k < k_{j+1}}))}.
 \end{equation}
Since $N \le 2^{l+2}\lambda$, it is clear from Lemma \ref{vnrmlemma} that the first term on the right of \eqref{longshortagain} is controlled by the right hand side of \eqref{localunlinearized}. 

We 'll show that $k_j \leq k < k_{j+1}$ then
\begin{equation}\label{insertdelta}
\sum_{ 2^k  \leq |I_P| < 2^{k_{j+1}}} \<f,\phi_{P_3}\>\phi_{P_3} = \Delta_{k-(\nu + 1)}[\sum_{2^{k_{j}} \leq |I_P| < 2^{k_{j+1}}} \<f,\phi_{P_3}\>\phi_{P_3}].
\end{equation}
Since $|\Xi|\le 2^{l+1}\lambda$, it follows from \eqref{insertdelta} and Lemma \ref{vblemma} that, for each $j$,
\begin{eqnarray*}
\|\sum_{ 2^k  \leq |I_P| < 2^{k_{j+1}}} \<f,\phi_{P_3}\>\phi_{P_3}(x)\|_{L^2_x(V^r_{k_{j} \leq k < k_{j+1}})} 
 \\ 
\leq C  (2^l \lambda)^{\epsilon} \|\sum_{2^{k_j} \leq |I_P| < 2^{k_{j+1}}} \<f,\phi_{P_3}\>\phi_{P_3}(x)\|_{L^2}. 
\end{eqnarray*}
Taking $\ell^2$ sum over $j$, the second term on the right of \eqref{longshortagain} is clearly controlled by the right hand side of \eqref{localunlinearized}.

It remains to show \eqref{insertdelta}. It is clear from the definition of $\nu$ that if $k \le \log_2 |I_P|$ then  $\Delta_{k-\nu-1}[\phi_{P_3}] = \phi_{P_3}$. Therefore it remains to show 
$$\Delta_{k-\nu-1}[\phi_{P_3}] =0$$ 
if $2^{k_j}\le |I_P|<2^k$. Assume that this is not the case. Then one interval in $\Omega_{k-\nu-1}$ must intersect $\omega_{P_3}$. Since $k> n:=\log_2|I_P|$, one interval in $\Omega_{n-\nu}$ must also intersect $\omega_{P_3}$. By definition of $\nu$, this interval and the dyadic interval of length $2^{\nu-n}$ containing $\xi_T$ are siblings, i.e. they share the same dyadic parent, here $T$ is the tree where $P$ lives. Thus,
$$|\Omega_{n-\nu}| > |\Omega_{n-\nu-1}|$$
therefore $|\Omega_{n+5}|>|\Omega_{n-3}|$. But $k_j< n+1 \le k < k_{j+1}$ so this violates the choice of $\{k_j\}$. \end{proof}

\section{Proof of Theorem \ref{maintheorem}} \label{poftsection}

Recall from \eqref{modelop} that our aim is to prove that the operator
\begin{equation} \label{linearizedundualized}
\sum_{P \in \ov{\P}} |I_P|^{-1/2}\<f_1,\tphi_{P_1}\>\<f_2,\tphi_{P_2}\>\tphi_{P_3}(x) 
\end{equation}
is bounded from $L^{p_1} \times L^{p_2}$ into $L^q$ whenever $\frac{2}{3} < q < \infty$ and $1 < p_1,p_2 \le \infty.$
Without loss of generality, we can assume that $\ov\P$ is finite,  provided that the estimates are uniform in $\ov{\P}.$ 

Despite the possibility that $q < 1$, the ``restricted type'' interpolation method of \cite{muscalu02mlo} allows one to deduce bounds on \eqref{linearizedundualized} from certain estimates for
\[
\Lambda(f_1,f_2,f_3) = \sum_{P \in \ov{\P}} |I_P|^{-1/2} \prod_{i = 1}^3\<f_i,\tphi_{P_i}\>.
\]
Specifically, our desired bounds follows from Proposition \ref{restrictedtypeproposition} and its symmetric variants (whose proofs are analogous). Below, we say $H \subset G$ is a major subset if $|H|\ge |G|/2$.

\begin{proposition} \label{restrictedtypeproposition}
Let $r > 2$ and  $E_1,E_2,E_3$ be subsets of $\rea^+$ of positive measures. Assume $|f_i| \le 1_{E_i}$ for every $i$. Then there exists a major subset $\tilde{E_1}$ of $E_1$ such that in any neighborhood of $(-\frac 1 2, \frac 1 2, 1)$ we can find an $\alpha=(\alpha_1,\alpha_2,\alpha_3)$ with $\alpha_1+\alpha_2 +\alpha_3=1$  satisfying
\[
|\Lambda(f_11_{\tilde{E}_1},f_2,f_3)| \leq C_{r,\alpha} \prod_{i=1}^3 |E_i|^{\alpha_i}.
\]
\end{proposition}

\begin{proof}
By (dyadic) dilation symmetry we can assume $|E_1| \in [1/2, 1)$. Fix $q>1$ close to $1$ to be chosen later. Then we choose $\tilde{E}_1 = E_1 \setminus F$ where
$$F = \bigcup_{i=1}^3 \{M^q[1_{E_i}] \geq C |E_i|^{1/q}\} $$
with $C$ is chosen sufficiently large to guarantee that $|F| \leq \frac{1}{4}$. 

Now, without loss of generality assume that $f_1$ is supported in $\tilde{E}_1$. Then the quartiles that contribute to $\Lambda$ belong to $\P = \{P \in \ov{\P} : I_P \not\subset F\}$.  Note that by Proposition \ref{propositionvnsb} we have
\begin{equation}\label{universalsizebound}
S_i:= \size_{i}(\P, f_i) \le C |E_i|^{1/q}.
\end{equation}
Here $C$ and other implicit constants below can depend on $r$, $q$, and $\beta_i$ (defined below).

Applying Proposition \ref{vnsprop} repeatedly, we obtain a decomposition of $\P$ into  collections of trees $(\T_n)_{n\in \mathbb Z}$ with
\begin{equation} \label{treeboundinproof}
\sum_{T \in \T_n} |I_T| \leq C 2^n,
\end{equation}
and furthermore for any $T\in \T_n$ we have
\begin{equation}\label{sizeboundinproof}
\size_{i}(T,f_i) \leq C 2^{-n/(2q)}|E_i|^{1/(2q)}.
\end{equation}

Now, for any tree $T$ we have
\begin{equation} \label{treeestimate}
\sum_{P \in T} |I_P|^{-1/2} \prod_{i = 1}^3|\<f_i,\tphi_{P_i}\>| \leq 4 |I_T|\prod_{i = 1}^3 \size_i(T,f_i).
\end{equation}
To show \eqref{treeestimate}, by further decomposing $T$ we can assume that it is $i$-overlapping for some $i \in \{1, 2,3, 4\}$. If $i \ne 4$ we will estimate for every $P \in T$
\[
|I_P|^{-1/2}|\<f_i,\tphi_{P_i}\>| \leq \size_i(T,f_i)
\]
and apply Cauchy-Schwarz to estimate the remaining bilinear sum by
$$|I_T| \prod_{j \in \{1,2,3\} \setminus \{i\}}\size_j(T, f_j).$$
The case $i =4$ is even simpler, one can apply the above $\ell^\infty \times \ell^2 \times \ell^2$ estimate in any order.

Applying \eqref{treeboundinproof}, \eqref{sizeboundinproof}, \eqref{treeestimate},  we obtain
\begin{align*}
|\Lambda(f_1, f_2, f_3)|   &\leq C\sum_{n} 2^n \prod_{i=1}^3 \min(S_i, 2^{-n/(2q)}|E_i|^{1/(2q)}).
\end{align*}
For any $\beta_1,\beta_2,\beta_3\in [0,1]$ we can further estimate by
$$ \le C\, S_1 S_2 S_3 \sum_{n} 2^n \min \Big(1, 2^{-n\frac{\beta_1+\beta_2 + \beta_3}{2q}} \prod_{i=1}^3  |E_i|^{\frac{\beta_i}{2q}} S_i^{-\beta_i} \Big).$$
The above estimate is a two sided geometric series if we choose $\beta_i$'s such that $\beta_1+\beta_2+\beta_3 > 2q$ (which is possible if $q$ is close to $1$). We obtain
$$|\Lambda(f_1, f_2, f_3)|   \le C\prod_{i=1}^3 S_i^{1-\gamma_i} |E_i|^{\gamma_i/(2q)} \ , \qquad \gamma_i := 2q\beta_i/(\beta_1+\beta_2+\beta_3),$$

$$\le C \Big(\prod_{i=1}^3 |E_i|^{1 -\frac{\gamma_i}{2}}\Big)^{1/q} \qquad \text{(using \eqref{universalsizebound} )}.$$
Since $|E_1| \sim 1$, we can ignore its contribution in the above estimate. Now, by sending $(q,\beta_1,\beta_2,\beta_3)$ to $(1,1,1,0)$ inside the region $\{\beta_1+\beta_2+\beta_3>2q\} \cap \{0\le \beta_1,\beta_2,\beta_3\le 1<q\}$, we obtain the desired claim.
\end{proof}

\bibliographystyle{amsplain}
\bibliography{variationaltree}
\end{document}